\documentclass{birkjour}
\usepackage{bm}

\newtheorem{theorem}{Theorem}[section]
\newtheorem{lemma}[theorem]{Lemma}
\newtheorem{corollary}[theorem]{Corollary}
\newtheorem{condition}[theorem]{Condition}
\newtheorem{Main}[theorem]{Main Theorem}

\theoremstyle{definition}
\newtheorem{definition}[theorem]{Definition}

\theoremstyle{remark}
\newtheorem{remark}[theorem]{Remark}

\newcommand{\R}{\mathbb{R}}
\renewcommand{\div}{{\rm div}}
\newcommand{\curl}{{\rm curl}}
\newcommand{\grad}{{\rm grad}}
\newcommand{\norm}[1]{{\left\|#1\right\|}}

\numberwithin{equation}{section}

\begin{document}

\title{Unique Measure for the Time-Periodic Navier-Stokes on the Sphere}

\author{Gregory Varner}
\address{Division of Natural and Health Sciences, Mathematics Department
John Brown University, Siloam Springs, AR 72761}
\email{gvarner@jbu.edu}



\date{\today}




\begin{abstract}
This paper proves the existence and uniqueness of a time-invariant measure for the 2D Navier-Stokes equations on the sphere under a random kick-force and a time-periodic deterministic force. Several examples of deterministic forces satisfying the necessary conditions for there to be a unique invariant measure are given. The support of the measure is examined and given explicitly for several cases.
\end{abstract}

\maketitle

\section*{Introduction}
The existence and uniqueness of a time-invariant measure for the Navier-Stokes equations has been the subject of much recent research. A major advance was achieved in \cite{Kuksin2} where it was shown that, under a random bounded kick-type force, the Navier-Stokes system on the torus (bounded domains with smooth boundaries and periodic boundary conditions) has a unique time-invariant measure. Subsequently, the argument was refined to a more flexible coupling approach in \cite{Kuksin}, which paved the way for extending the argument to the case of a white-noise random force (\cite{Kuksin3}, \cite{Kuksin4}, or \cite{Shirikyan}). Except for the white-noise case, these methods focused on the case of zero deterministic forcing on the system and the equations on the torus. Of course, for meteorological purposes, it is desirable to consider the equations on the sphere and to require the deterministic force to be nonzero. This was done in \cite{Varner}, where a time-invariant measure for the Navier-Stokes equations on the sphere was shown to exist both under a random bounded kick-type force with a time-independent deterministic force and under a white-noise force. 

In this paper the work in \cite{Varner} is extended to include time-periodic deterministic forces. A similar result was established in \cite{Shirikyan4} for the torus and with a random perturbation activated by an indicator function. We instead use a random perturbation activated by a dirac function and include the more general case of a squeezing-type property. Even though the random force in \cite{Shirikyan4} is more general in the sense of time-dependence, we use the random kick-force as in \cite{Kuksin} and \cite{Varner} to allow less regularity assumptions and to highlight the similarities between the time-independent and time-periodic cases. 

The first section uses a combination of the approaches in \cite{Ilin}, \cite{Ilin5} \cite{Titi}, and \cite{Brzezniak} to define each of the terms in the Navier-Stokes equations on the sphere. Of utmost importance are the eigenvalues of the Laplacian term which allow the analysis to proceed as in the case of flat domains. In addition, we consider the Navier-Stokes equations under time-periodic forcing, establishing conditions for there to be a limiting solution that is periodic, including several cases where the period of the unique solution is the same as the force.

The second section presents the main theorem, which establishes the existence and uniqueness of an invariant measure for the kicked equations with a time-periodic deterministic external force. The proof of the main theorem is done by proving that necessary conditions hold for the applicability of Theorem 3.2.5 in \cite{Kuksin7}. As will be seen, the periodicity of the deterministic force allows the argument for stationary forces to be applied to the time-periodic case. The necessary conditions for the main theorem are shown for several cases including a contraction-type property and a squeezing-type property with ``large" random kicks. The main idea behind the contraction-type property is the exponential stability of solutions, i.e., the contraction of the flow to a unique solution, while the squeezing-type property is related to the idea of determining modes (\cite{Robinson}, p. 363). More precisely, if the projection of the initial conditions onto the first $M$ eigenfunctions is close enough then the solutions will converge.  
 
The third section recalls work done in \cite{Varner} describing the support of the measure. The support is described both in general and specifically for several examples. By combining results in \cite{Varner} and \cite{Heywood}, the support of the measure is also described in terms of a unique time-periodic solution in several cases, including some of potential meteorological interest.

\section{The Navier-Stokes Equations on the Sphere}

Let $M=S^{2}$ be the 2-dimensional sphere with the Riemannian metric induced from $R^{3}$. Let $\left(\phi,\lambda\right)$ be the spherical coordinate system on $M$, where $\phi\in\left(-\frac{\pi}{2},\frac{\pi}{2}\right)$ is the co-latitude (the geographical latitude) and $\lambda\in\left(0,2\pi\right)$ is the longitude. 
Furthermore, $\vec{n}=\left(\cos\phi\cos\lambda,\cos\phi\sin\lambda,\sin\phi\right)$ is the outward normal to $M$ in $R^{3}$. 
Let $H_{\phi}=\left|\frac{\partial\vec{n}}{\partial\phi}\right|$ and $H_{\lambda}=\left|\frac{\partial\vec{n}}{\partial\lambda}\right|$, then the unit vectors 
$$\vec{\phi}=\frac{1}{\left|H_{\phi}\right|}\frac{\partial\vec{n}}{\partial\phi} \quad \vec{\lambda}=\frac{1}{\left|H_{\lambda}\right|}\frac{\partial\vec{n}}{\partial\lambda}$$ 
form a basis for the tangent space of $M$, denoted $TM$ and induce on $M$ the Reimannian metric
$$\left(g_{ij}\right)=
\left(\begin{array}{cc}
	1 & 0 \\
	0 & \cos^{2}\phi
\end{array}\right).
$$  
Note that a vector tangent to $M$ can be decomposed as $u=u_{\phi}\vec{\phi} + u_{\lambda}\vec{\lambda}$ (unless there is a danger of ambiguity, we use the same notation for expressing functions and vector fields. The main exception are the normal and unit vectors, which are expressed as $\vec{n}$, $\vec{\phi}$, and $\vec{\lambda}$).

 The Navier-Stokes equations on the rotating sphere are
\begin{equation}\label{deterministicnavierstokes}
	\begin{split}
		&	\partial_{t}u
			+
			\nabla_{u}u
			-
			\nu\Delta u
			+
			l\,\vec{n}\times u
			+
			\nabla \ p 
			=
			f,  \\
		& \div \ u =0, 	\ u|_{t=0}=u_{0},
	\end{split}
\end{equation}
where $\vec{n}$ is the normal vector to the sphere, $l=2\Omega\sin\phi$ is the Coriolis coefficient, $\Omega$ is the angular velocity of the Earth, and ``$\times$" is the standard cross product in $\R^{3}$.

The operators $\div$ and $\nabla$ in \eqref{deterministicnavierstokes} have their conventional meanings on the sphere, i.e. for functions $\psi$ and vectors $u$

\begin{equation}
	\nabla \psi = \frac{\partial \psi}{\partial \phi}\vec{\phi}+\left(\frac{1}{\cos\phi}\frac{\partial \psi}{\partial \lambda}\right)\vec{\lambda},
	\quad
	\div u = \frac{1}{\cos\phi}\left(\frac{\partial}{\partial \lambda}u_{\lambda}+\frac{\partial}{\partial \phi}\left(u_{\phi}\cos\phi\right)\right). \nonumber
\end{equation}

To define the covariant derivative $\nabla_{u}u$ and the vector Laplacian $\Delta$ we first define the curl of a vector in terms of extensions. For any covering $\left\{O_{i}\right\}$ of $M$ by open sets, there is a corresponding set of ``cylindrical domains" $\widetilde{O}_{i}$ that cover a tubular neighborhood of $M$, $\widetilde{M}$. In each $\widetilde{O}_{i}$ we introduce the orthogonal coordinate system $\widetilde{x}_{1}, \ \widetilde{x}_{2}, \ \widetilde{x}_{3}$, where $-\epsilon<\widetilde{x}_{3}< \epsilon$ is along the normal to $M$ and for $\widetilde{x}_{3}=0$ the coordinates $x_{1},x_{2}$ agree with the spherical coordinates.

For a vector $u\in TM$ there is a vector $\widetilde u$ defined in $\widetilde{M}$ such that the restriction to $M$ satisfies $\widetilde{u}|_{M}=u\in TM$. For a vector field $w$ on the sphere, not necessarily tangent to it, the curl of $w$ is a vector field along the sphere defined as (\cite{Ilin}, p. 562)
\begin{equation}
	Curl\, w := Curl\,\widetilde{w}|_{M}. \nonumber
\end{equation}

\noindent For a vector field normal to $M$, the curl is well-defined and is tangent to $M$. However, for a vector field in $TM$ the curl is not well-defined but the third component of the curl, which we denote as $\curl_{n}$, is well-defined. Due to this, define the following operators (\cite{Titi}, p. 344).

\begin{definition}
Let $u$ be a smooth vector field on $M$ with values in $TM$ and let $\vec{\psi}$ be a smooth vector field on $M$ with values in  $ TM^{\perp}$, i.e. $\vec{\psi}=\psi\vec{n}$ for $\psi$ a smooth scalar function. Therefore, we identify the vector field $\vec{\psi}$ with the function $\psi$. Denote the extensions $\widetilde{u}$ and $\widetilde{\psi}$. Then for $x\in M, \ y\in \R^{3}$ define
$$\curl\,\vec{\psi}(x) =\curl\,\psi= Curl\widetilde{\psi}(y)|_{y=x}$$
$$\curl_{n}u(x) = \left(Curl\widetilde{u}(y)\cdot\vec{n}(y)\right)\vec{n}(y)|_{y=x},$$
where on the right side $Curl$ denotes the standard $\curl$ operator in $\R^{3}$.

\noindent Note that these definitions are independent of the extensions - see \cite{Ilin}, p. 562.
\end{definition}

The covariant derivative and vector Laplacian are now defined in terms of the $\curl$ and $\curl_{n}$ operators (\cite{Ilin}, p. 562-563).

\begin{definition} The covariant derivative on the sphere is given by
	\begin{equation}
		\nabla_{u}u := \nabla\frac{\left|u\right|^{2}}{2}
	-
	u\times \curl_{n}u.
	\end{equation}
\end{definition}

\begin{remark}
As with the curl and $\curl_{n}$ operators, it is possible to define the gradient, divergence, and covariant derivative in terms of extensions (see \cite{Ilin} or \cite{Brzezniak}). However (\cite{Titi}, p. 344),
$$\curl\,\psi = -\vec{n}\times\nabla\psi, \quad \curl_{n}v=-\vec{n}\div\left(\vec{n}\times v\right).$$
Thus both $\curl$ and $\curl_{n}$, and thus  the gradient, divergence, and covariant derivative, can be defined without resorting to extensions.
\end{remark}

\begin{definition} The vector Laplacian on the sphere is given by (\cite{Ilin}, p. 563)
	\begin{equation}
		\Delta u := \nabla\,\div\, u - \curl\,\curl_{n}u.
	\end{equation}
\end{definition}


\noindent Thus, the Navier-Stokes system on the two-dimensional sphere, i.e., for vector fields on $M$, is:  
\begin{equation}
	\begin{split}
		&	\partial_{t}u
			+
			\nabla\frac{\left|u\right|^{2}}{2}
			-
			u\times\curl_{n}u
			+
			\nu\curl\curl_{n}u
			+
			l\,\vec{n}\times u
			+
			\nabla\, p
			=
			f, \\
		& \div u=0, \ u|_{t=0}=u_{0}.
	\end{split}
\end{equation}


\subsection{Existence and Uniqueness for the Deterministic Equations}

Let $L^{p}(M)$ and $L^{p}(TM)$ be the standard $L^{p}$-spaces of the square integrable scalar functions and tangent vector fields on $M$, respectively. The inner products for $L^{2}(M)$ and $L^{2}(TM)$ are given by:
  $$\left(u,v\right)_{L^{2}(M)}:= \int_{M}uv dM, \ u,v \in L^{2}(M), $$
	$$\left(u,v\right)_{L^{2}(TM)}:= \int_{M}u\cdot v dM, \ u, v \in L^{2}(TM). $$
	
	\noindent Note that these are integrals over oriented manifolds and thus are defined intrinsically using a partition of unity. Locally, however, $dM=\cos\phi d\phi d\lambda$. The induced norm on $L^{2}$ will be denoted $\norm{\cdot}_{L^{2}}$. 

\noindent Let $\psi$ be a scalar function and $v$ be a vector field on $M$. For $s\geq 0$, the standard Sobolev spaces $H^{s}(TM)$ have norm 
$$
	\norm{\psi}^{2}_{H^{s}(M)}:= \norm{\psi}^{2}_{L^{2}(M)}
	+
	\left\langle -\Delta^{s} \psi, \psi\right\rangle_{L^{2}(M)}$$
and

$$
	\norm{u}^{2}_{H^{s}(TM)}:= \norm{u}^{2}_{L^{2}(TM)}
	+
	\left\langle -\Delta^{s} u, u\right\rangle_{L^{2}(TM)}.$$



\noindent Since the sphere is simply connected, the Hodge Decomposition Theorem gives that the space of smooth vector fields on $M$ can be decomposed as (\cite{Ilin}, p.564):
\begin{eqnarray}
		C^{\infty}(TM) \nonumber \\
		= &\left\{u: u=\grad\phi, \phi\in C^{\infty}(M)\right\}\oplus
									\left\{u: u=\curl\phi, \phi\in C^{\infty}(M)\right\} \nonumber \\
	 = &\left\{u: u=\grad\phi, \phi\in C^{\infty}(M)\right\}\oplus
									V_{0}. \nonumber						
\end{eqnarray}

\noindent Define the following closed subspaces of $L^{2}(TM)$ and $H^{1}(TM)$ respectively:
\begin{definition}  
\begin{equation}
	H := 
	\curl(H^{1}(M)), \nonumber
\end{equation}
 with norm
\begin{equation}
	\norm{u}_{H} = 
	\norm{u}_{L^{2}(TM)}.
\end{equation}
$H$ is the $L^{2}$ closure of $V_{0}$ and thus $\div\,u=0$ for $u\in H$.
\end{definition}

\begin{definition}
\begin{equation}
	V := 
	\curl(H^{2}(M)) \nonumber
\end{equation}
with norm 
\begin{equation}
	\norm{u}_{V} = 
	\norm{\curl_{n}u}_{L^{2}(TM)}.
\end{equation}
$V$ is the $H^{1}$ closure of $V_{0}$ and thus $\div\,u=0$ for $u\in V$. Furthermore, $V$ is compactly embedded into $H$, and by the Poincare Inequality (equation \eqref{Poincare})  the $V$ norm is equivalent to the $H^{1}$ norm for divergence-free vector fields.
\end{definition}

\begin{definition} For a vector field $u$, define the Laplacian on divergence-free vector fields as
	\begin{equation}
		Au:= \curl\curl_{n}u,
	\end{equation}
	where if $\div\,u=0$ then $Au= -\Delta u$. 
\end{definition}

The following theorem implies that the analysis used for the stochastic Navier-Stokes system on flat domains can be used for the system on the sphere. Its proof is identical to the case of flat domains with smooth boundary conditions, see \cite{Robinson}, pp. 162-163 or \cite{Ilin}, p. 565.

\begin{theorem} The operator $A=\curl\curl_{n}$ is a self-adjoint positive-definite operator in $H$ with eigenvalues $0< \lambda_{1}\leq \lambda_{2}\leq ... $ with the only accumulation point $\infty$. Moreover, the eigenvalues correspond to an orthonormal basis in H (orthogonal in V).
\end{theorem}

Let $P_{H}$ be the projection onto $H$. Since the projection commutes with $\partial_{t}$ and $A$, the projection of \eqref{deterministicnavierstokes} onto $H$ is 
\begin{equation}
	\partial_{t}u 
	+
	\nu Au
	+ 
	B(u,u) 
	+
	C(u)
	= 
	f \label{projectednavierstokes}
\end{equation}
where $B(u,u)+C(u)=P_{H}(\nabla_{u}u+ l\,\vec{n}\times u)$. Furthermore, for all $v\in V$
\begin{equation}
 	\left\langle B(u,u) + C(u),v\right\rangle_{H} 
 	= 
 	b(u,u,v) + \left\langle l\,\vec{n}\times u,v\right\rangle_{H}, 
\end{equation}
where $b(u,v,w)$ is the standard trilinear form associated with the Navier-Stokes equations, i.e. 
\begin{equation}
	b(u,v,w) = \pi \sum_{i,j=1}^{3}\int_{M}u_{j}D_{i}v_{j}w_{j}dx,
\end{equation}
where $\pi$ is the orthogonal projection onto $TM$ (\cite{Ilin}, p. 561), and the trilinear terms satisfies estimates analogous to those in the case of flat domains, see Lemma \ref{trilinearestimateslemma}.


We now state the existence and uniqueness of solutions to the deterministic Navier-Stokes equations in terms of the projected equations, as is standard.


\begin{theorem}\label{existenceofprojectedsolution} Suppose that $f \in  L^{2}(0,T; H)$ and $u_{0}\in H$ then a solution of the Navier-Stokes equations with Coriolis \eqref{projectednavierstokes} exists uniquely and $u\in L^{2}(0,T;V)\cap C([0,T];H)$. If $u_{0}\in V$ then the unique solution is strong, i.e. $u\in L^{2}(0,T;D(A))\cap C([0,T];V)$ and $\dfrac{du}{dt} \in L^{2}(0,T; H)$. 
\end{theorem}
\noindent The proof is the same as the case of bounded domains with smooth boundaries and periodic boundary conditions (see \cite{Robinson}, pp. 245-254 and \cite{Ilin}, Theorem 2.2).

\subsection{Time-Periodic Navier-Stokes Equations on the Sphere}

Assume $f\in L^{\infty}(0,\infty; H)$ (thus in $L^{2}(0,T;H)$ for any $T<\infty$) is periodic with period $T>0$. While Theorem \ref{existenceofprojectedsolution} gives the existence of a strong solution to the Navier-Stokes equations on the sphere, it will be necessary to know the behavior of the system under a periodic force. Toward that end, we recall a theorem from \cite{Heywood}, p. 19. For the theorem, the following definition will be needed.

\begin{definition} Let $w=u-v$ be a perturbation of the solution $u$ of the Navier-Stokes equations. $u$ is  {\bfseries exponentially stable} if there exist numbers $\delta, \alpha, A>0$ such that every perturbation at time $t_{0}$, with $w_{0}=w(t_{0})$ and $\norm{w_{0}}_{H}<\delta$ satisfies
\begin{equation}
	\norm{u(t)-v(t)}_{H}=\norm{w(t)}_{H}\leq Ae^{-\alpha(t-t_{0})}\norm{w_{0}}_{H}, \ \text{for \ all} \ t\geq t_{0}.
\end{equation}
	
\noindent	$\delta$ is called the {\bfseries stability radius}. If $\delta=\infty$, a solution is called {\bfseries globally exponentially stable}.
\end{definition}

\begin{theorem}\label{Heywoodtheorem}
	Suppose there exists a globally defined solution to the Navier-Stokes equation with initial condition in $H$, has $\norm{S_{t}u_{0}}_{H^{1}}$ bounded, and is exponentially stable. If $f$ is time-periodic with period T then there exists a time-periodic solution $u_{\infty}$ with period $kT$ for some integer k, such that
\begin{equation}
	\norm{S_{t}u_{0}-S_{t}u_{\infty}}_{H^{1}} = O(e^{-\alpha t}), \ \text{some} \ \alpha>0, \ \text{as} \ t\rightarrow \infty. 
\end{equation} 

\noindent If the stability radius $\delta$ is large enough or the period is small enough then $T_{\infty}=T$. In all cases, $u_{\infty}$ is exponentially stable. 
\end{theorem}

While the theorem in \cite{Heywood} assumes that the initial condition is in $H^{1}$,  this is for $\norm{S_{t}u_{0}}_{H^{1}}$ to be bounded. By Theorem \ref{existenceofprojectedsolution} (or equation \eqref{VtoHcontraction}) this norm is bounded for $u_{0}\in H$ for any $t>0$. Furthermore, Theorem \ref{Heywoodtheorem} implies convergence in $H$ and the proof in \cite{Heywood} is easily adapted to show exponential convergence in $H$ instead.

It is well-known that if the force is small enough (see Remark \ref{differenceofsolutionssatisfiescontraction}) then the stability radius is infinite. We conclude this section by examining two more cases where the stability radius is infinite, i.e. there is a unique globally exponentially stable solution with the same period as the force. The proofs of the following lemmas are found in the Appendix following necessary estimates. The main idea behind both of the lemmas is that the (spherical) scalar Laplacian commutes with longitudinal derivatives, allowing for terms in the calculations only dependent on latitude to vanish.

\begin{definition} A solution to the Navier-Stokes equations, $u$, is called zonal if for each fixed $t$, $u(t)$ is only a function of latitude, i.e. the function has no longitudinal dependence.
\end{definition}

\begin{lemma}\label{thiszonallemma} Suppose that the time-periodic force $f\in L^{\infty}(0,\infty;H)$ is such that there is a zonal solution. Then the solution is unique with the same period as $f$.
\end{lemma} 

\begin{remark} 
For a stationary force, it is sufficient that the force is zonal to have a stationary zonal solution (\cite{Ilin5}, p. 988) which follows since $A$ forms an isomorphism between the spaces $D(A)$ and $H$ and for $u$ zonal
\begin{equation}\nonumber
\norm{B(u,u)+C(u)}_{H}\leq C\norm{B(u,u)+C(u)}_{V}=0.
\end{equation}

Analogously, the Stoke's equation $\partial_{t}u+\nu Au$ forms an isomorphism between the spaces
\begin{eqnarray} 
 & \left\{u, \ u\in L^{2}(0,T;D(A)), \ u'\in L^{2}(0,T;H), \ u(0)=u(T)\right\} \nonumber \\
 & \text{and} \ L^{2}(0,T;H). \nonumber
\end{eqnarray}
Thus, to have a zonal solution it is sufficient that force is zonal. (The proof that the equations form an isomorphism is analogous to the result in \cite{Furshikov}, Lemma 3.1, p. 27 or \cite{Lions}, Chapter 4, Section 15.)
\end{remark}

\begin{lemma}\label{AlmostZonalProp} Let $f\in L^{\infty}(0,\infty;H)$ be a force that generates a zonal solution. There exists $\delta >0$ such that for any $g\in  L^{\infty}(0,\infty;H)$ satisfying $\norm{f-g}_{L^{\infty}(0,\infty;H)} < \delta$ there is a unique globally exponentially stable solution to the Navier-Stokes equations.
\end{lemma}

\begin{definition}
We define an {\bfseries almost zonal solution} to be a solution guaranteed by Lemma \ref{AlmostZonalProp}.
\end{definition}
 
It is worth noting that while Lemma \ref{AlmostZonalProp} allows for nonzonal solutions, they are only a ``small'' perturbation from being zonal.


\section{The Main Theorem}

This section presents the main theorem on the existence and uniqueness of a (time-)invariant measure for the Navier-Stokes system with random kicks and a time-periodic deterministic force, where time-invariance is understood to mean that the random variables generated by restricting the solutions to instants of time proportional to the period of the deterministic forcing term have a unique stationary probability distribution which all other distributions converge to exponentially (i.e. it is exponentially mixing).  A similar result in \cite{Shirikyan4} established that the Navier-Stokes equations on the torus have an unique invariant measure under a deterministic time-periodic forcing. While the random force considered in \cite{Shirikyan4} allows for more generality in the sense of time-dependence, the random force also requires additional regularity. We instead use the bounded random kick-force as in \cite{Kuksin} and \cite{Varner} to highlight the similarities to the time-independent case and due to weaker assumptions on the regularity of the random force. In particular, we use Theorem 3.2.5 in \cite{Kuksin7} which focuses on the properties of the solution operator and the perturbed flow, which are not as explicit in \cite{Shirikyan4}. In addition, we consider cases of potential meteorological interest and more general deterministic forces.

\subsection{The Perturbed Navier-Stokes Equations}
Consider the Navier-Stokes system with forcing $f\in L^{\infty}\left(0,\infty,H\right)$ time-periodic with period $T$, and a random kick-force $g$ bounded in $H$:
\begin{equation} \label{kicked}
\begin{split}
	&\partial_{t}u
	+
	\nu Au
	+
	B(u,u)
	+
	C(u)
	=
	f+ g,\\    
	&g = 
	\sum_{k=1}^{\infty}
	\eta_{k}(x)\delta_{kT}(t),
	 \ \eta_{k}\in H, 
	 \  \norm{\eta_{k}}_{H}< \infty \ \forall k.
\end{split}
\end{equation}

The notation from now on will be:
\begin{itemize}
	\item $S_{t}v_{0}$ is the solution of the deterministic equation with initial condition $v_{0}\in H$ at time $t\geq 0$.
	\item For simplicity of notation take the period as $T=1$ and denote $S_{1}=S$. 
	\item $u^{t}(v_{0})$ is the solution of \eqref{kicked} with initial condition $v_{0}$ at time $t\geq 0$. 
\end{itemize}

\noindent Then
\begin{equation}
\begin{split}
	& u^{0}(v_{0})= v_{0} \\
  & u^{k+1}(v_{0})
  = Su^{k}(v_{0})+\eta_{k+1}(x),  \ k=0,1,2,\ldots \\
  & u^{k+\tau}(v_{0}) = S_{\tau}u^{k}(v), 
   \ 0\leq\tau < 1, 
  \ k=0,1,2,\ldots
\end{split}
\end{equation}
In other words, the solution between kicks is given by the flow of the deterministic system with time-periodic forcing. Notice that due to the periodicity of the force, if all the kicks were zero then for any positive integer $n$, $S_{n}v_{0}= u^{n}(v_{0}).$ 

Following \cite{Kuksin}, pp. 356-357, assume the kicks satisfy:

\begin{condition} \label{mainkicksassumption}
Let $\left\{e_{j}\right\}$ be the orthonormal basis for the Hilbert space $H$, 
then 
		\begin{equation} 
		 	\eta_{k} = \sum_{j=1}^{\infty}b_{j}\zeta_{jk}e_{j}, \ b_{j}\geq 0, \ B_{0}=
		 	 \sum_{j=1}^{\infty}b_{j}^{2}<\infty, 
		\end{equation}
for $\left\{\zeta_{jk}\right\}$ a family of independent, identically distributed real-valued variables, with $\left|\zeta_{jk}\right|\leq 1$ for all $j,k$. 
Their common law has density $p_{j}$ with respect to Lebesgue measure where $p_{j}$ is of bounded variation with support in the interval $\left[-1,1\right]$.
Furthermore, for any $\epsilon>0$, $\displaystyle \int_{\left|r\right|<\epsilon}p_{j}(r)dr>0.$  
\end{condition}


For a given positive integer $k$ and $v_{0}$, the Markov transition measure 
$\beta(k,v_{0}, .)$ is defined as
\begin{equation}
    \beta(k,v_{0},\Gamma)= \mathbb{P}\left\{u^{k}(v_{0})\in \Gamma\right\}, \ k\geq 0, \ v_{0}\in H, \ \Gamma\in \mathcal{B}(H), \nonumber
\end{equation}
where $\mathcal{B}(H)$ is the Borel $\sigma$-algebra of $H$. The Markov transition measure is the probability that the stochastic flow with initial condition $v_{0}$ is in the set $\Gamma$ at time $k$, i.e. $u^{k}(v_{0})_{\#}\mathbb{P}$. 

 
The Markov semigroup $\beta_{k}$ 
on bounded continuous functions is defined by
\begin{equation}
    \beta_{k}h(v) = \mathbb{E}h(u^{k}(v)) = \int_{H}h(z)\beta(k,v,dz), \nonumber
\end{equation}
where $h:H \to \R$ is a 1-periodic bounded continuous function.

\begin{definition}
A measure $\mu \in P(H)$ is called invariant 
if $\beta_{k}^{*}\mu = \mu$ where $P(H)$ is the space of probability measures on $H$ and
\begin{equation}
    \beta_{k}^{*}\mu(\Gamma)= \int_{H}\mathbb{P}\left\{u^{k}(v)\in \Gamma\right\}\mu(dv), \ \Gamma\in \mathcal{B}(H). \nonumber
\end{equation}
\end{definition}

The next two definitions deal with behavior of the deterministic flow and are necessary for the statement of the main theorem.

\begin{definition}
We say that there is an {\bf asymptotically stable solution} if for some $q<1$, for all $R>0$, and for all $t\geq 0$
\begin{equation} \label{FForceCondition}
    \norm{S_{t}u_{0}-S_{t}v_{0}}_{H}
    \leq
    C(R)q^t
    \norm{u_{0}-v_{0}}_{H}
    \ \forall \ u_{0},v_{0}\in B_{H}(R)
\end{equation}
where $C(R)$ can depend on the norm of the force and $B_{H}(R)$ is the ball of radius $R$ centered at 0 in $H$.
\end{definition}
\noindent An asymptotically stable solution, is also globally exponentially stable for any radius $\delta >0$.

Note that the following satisfy condition \eqref{FForceCondition}:
\begin{itemize}
\item
$f=0$ and ``small'' forces, see Remark \ref{differenceofsolutionssatisfiescontraction}.
\item
Time-periodic forces that give zonal flow, see Lemma \ref{thiszonallemma}.
\item
Time-periodic forces that give almost zonal flow, see Lemma \ref{AlmostZonalProp}.
\end{itemize}

 Since the Navier-Stokes equations have an absorbing set (\cite{Ilin}, p. 572) any asymptotically stable solution is in a ball of finite radius in $H$, call it $D(f)$. In addition, an asymptotically stable solution guarantees that two deterministic solutions with different initial conditions will become arbitrarily close together as $t\rightarrow \infty$. In the same way, any point that locally acts like an asymptotically stable solution will be a (local) contraction of the flow and should be considered. However, due to the nature of the kicks, it will be possible to assume that only the first finitely many dimensions must have $b_{j}>0$. This implies that the solution only needs to be locally stable in a finite number of dimensions.

\begin{definition}\label{finitelystable} Let $D(f)$ be the radius of the deterministic absorbing set, $R>0$, and $P_{M}$ be the projection onto the first $M$ eigenfunctions. A point $u\in B_{H}(D(f))$ is called {\bf finitely stable} if for some $M\geq 1$, for some $\delta>0$, and for all $v\in B_{H}(R)$ satisfying $\norm{P_{M}u-P_{M}v}_{H}\leq \delta$, 
\begin{equation}
	\norm{S_{t}u-S_{t}(v)}_{H}\rightarrow 0. 
\end{equation}
In other words, if the finite-dimensional projections are ``close enough,'' then the solutions converge.
\end{definition}

Notice that a finitely stable point satisfies the conditions of Theorem \ref{Heywoodtheorem} (the stability radius and $\delta$ from Definition \ref{finitelystable} can be taken the same) and captures the same concept as determining modes (\cite{Robinson}, page 363). Furthermore, if $\delta$ is large enough relative to $T$ then the periodic solution converged to has period $T$. 

While the assumption of a finitely stable point allows for the possibility of multiple solutions, this also means that additional assumptions are needed for the structure of the kicks.

\begin{definition} The following is called the {\bf big kick assumption.} 
Let $M$ be as in Definition \ref{finitelystable}. For some $N\geq M$ let the $b_{i}$ from Condition \ref{mainkicksassumption} satisfy
\begin{align}
	& b_{1}\geq 2D(f),\nonumber \\
	& b_{j}\geq \dfrac{2D}{\lambda_{j}^{1/2}} \ \text{for} \ 2\leq j \leq M, \\ \label{brandnewassumption} 
	& b_{j}>0 \ \text{for} \ M<j\leq N . \nonumber
\end{align}
where $D=D(f)$ is the same as in \eqref{Condition1b} and $\lambda_{j}$ is the eigenvalue corresponding to $e_{j}(x)$.
\end{definition}

Notice that by equations \eqref{Vcontraction} and \eqref{Poincare} the $b_{j}$ are assumed to be twice as large as $\norm{Q_{j-1}u(t)}_{H}$ if the initial condition is zero (where $Q_{n}=I-P_{n}$). If the stochastic flow is within $\delta$ of the ball of radius $D(f)$, the kicks are large enough to ``kick'' the first $M$-dimensions of the flow within $\delta$ of the first $M$ dimensions of any point, in particular a finitely stable point, in the deterministic absorbing ball with nonzero probability.  Thus intuitively, the big kick assumption allows the perturbation to ``kick" the flow from anywhere in the absorbing ball into the stability radius of a finitely stable point.

\begin{Main} \label{THEMAINTHEOREM}
Let the kicks satisfy Condition \eqref{mainkicksassumption} and assume $f\in L^{\infty}(0,\infty;H)$ is time-periodic with period $T=1$, and that either:
\begin{itemize}
	\item there exists at least one finitely-stable point and the
big kick assumption holds or 
	\item there is an asymptotically stable solution.
\end{itemize}
Then there is $N$ such that if
$b_{j}>0$ for $j=1,2,...,N$ the following hold:
\begin{enumerate}
\item
The system \eqref{kicked}
has invariant measure $\mu$.
\item
The invariant measure is unique.
\item
For any $R>0$ there is $C(R,f)>0$ such that for any $h$ 1-periodic real-valued Lipschitz function on $H$
\begin{equation}
    \left|\beta_{k}h(u)-\left(\mu,h\right)\right|
    \leq C(R,f)e^{-ck}\norm{h}_{L} \ \text{for} \ k\geq 0, \ \forall \norm{u}_{H}\leq R.  \nonumber
\end{equation}
The constant $c>0$ is a constant not dependent on
$h, \ u, \ R, \ \text{or} \ k$.
\end{enumerate}
\end{Main}

\subsection{Proof of the Main Theorem}

The main theorem will follow from applying a modified version of Theorem 3.2.5 in \cite{Kuksin7}. Assume the following conditions.

\begin{condition} \label{generalcontraction}
For any $R$ and $r$ such that $R> r>0$ there exist $C= C(R,f), \ D=D(f), \ a= a(R,r)< 1$ all positive and there exists an integer $n_{0}=n_{0}(R,r)\geq 1$ such that
         \begin{align}
             & \norm{S_{n}u_{0}}_{H} \leq \max\left\{a\norm{u_{0}}_{H}+ D,r + D\right\}, \ u_{0}\in B_{H}(R), \ \forall n\geq n_{0} \label{Condition1b}, \\
             & \norm{Su_{0}-Sv_{0}}_{H} \leq C \norm{u_{0}-v_{0}}_{H},
              \ \forall u_{0},v_{0}\in B_{H}(R); \label{Condition1a}
         \end{align}
where $\norm{\eta_{k}}^{2}\leq B_{0}$ for all $k$.
\end{condition}

\begin{condition}\label{conditionforconvergence}
 For any $R >0$ there is a decreasing sequence $\gamma_{N}(R,f)>0, \ \gamma_{N}\rightarrow 0$ as $N\rightarrow \infty$ such that
        \begin{equation}
            \norm{(I-P_{N})(Su_{0}-Sv_{0})}_{H} \leq \gamma_{N}(R,f)\norm{u_{0}-v_{0}}_{H}, \ \forall
             u, v\in B_{H}(R),  \label{Condition2}
        \end{equation}
        where $P_{n}$ is the projection onto the first $N$ eigenfunctions $e_{j}$.
\end{condition}

\noindent Assume the kicked flow also satisfies:

\begin{condition}\label{neededforabsorbingset}
For $K$, the support of the distribution of $\eta_{k}$,
\begin{equation} \nonumber
\begin{split}
   K
   :=
   \left\{ u=\sum_{j=1}^{\infty}u_{j}e_{j}: \left|u_{j}\right|\leq b_{j} \ \forall j\geq 1\right\}
\end{split}
\end{equation}
and for any $B$ bounded in $H$ let
\begin{equation}
\begin{split}
   &A_{0}(B):= B, \\
   &A_{k}(B):= S(A_{k-1}(B))+K, \ k\geq 1.
\end{split}
\end{equation}
Then there exists $\rho>0$ such that for any $B$ there is $k_{0}(B,\rho)\geq 1$ such that:
\begin{equation}   \label{AbsorbingSet}
\begin{split}
     k \geq k_{0}
     \Rightarrow
     A_{k}(B) \subset B_{H}(\rho) .
\end{split}
\end{equation}
\end{condition}


\noindent In addition, assume that the kicked flow satisfies the following type of controllability.
\begin{condition}\label{controllability}
For any $d>0$ and $R>0$ there exists integer $l=l(d,R)>0$ and real number $x=x(d)>0$ such that
        \begin{equation}
 \mathbb{P}\left\{\|u^{l}(v_{0})-u^{l}(w_{0})\|_{H}\leq d\right\} \geq x, \ \text{for \ all} \ v_{0},w_{0}\in B_{H}(R).\label{Condition3}
        \end{equation}
\end{condition}

\noindent In other words, the kicked flow from two different initial conditions has a positive probability of becoming arbitrarily close together in finite time.

We now formulate a modified version of Theorem 3.2.5 from \cite{Kuksin7}.

\begin{theorem}  \label{CTheorem}
If the forced-kicked system \eqref{kicked} satisfies Conditions \ref{mainkicksassumption}, \ref{generalcontraction}, \ref{conditionforconvergence}, \ref{neededforabsorbingset}, and \ref{controllability}
then there is $N\geq 1$ such that if $b_{j}>0$ for all $1\leq j\leq N$, there exists an unique invariant measure, and
for any $R>0$ there is $C(R,f)>0$ such that for any 1-periodic real-valued Lipschitz function $h$ on $H$
\begin{equation}
    \left|\beta_{k}h(u)-\left(\mu,h\right)\right|
    \leq C(R,f)e^{-ck}\norm{h}_{L} \ \text{for} \ k\geq 0, \ \forall \norm{u}_{H}\leq R. \label{exponentialconvergence} \nonumber
\end{equation}
The constant $c>0$ is a constant not dependent on
$h, \ u, \ R, \ \text{or} \ k$. ($\norm{\cdot}_{L}$ is the standard Lipschitz norm.) 
\end{theorem}

\noindent While Theorem \ref{CTheorem} can be proved using the same approach as in \cite{Shirikyan4}, we instead use an approach similar to that in \cite{Kuksin}, \cite{Kuksin7}, and \cite{Varner}. While the proof in \cite{Kuksin7} has no external force $f$ and the one in \cite{Varner} only allows a time-independent force (which is periodic in time for any period), there are only two main differences in the above conditions and the ones used in \cite{Kuksin7}: the inequalities now depend on the norm of $f$ and the use of Condition \ref{controllability}. Due to these slight differences, a brief sketch of the proof of Theorem \ref{CTheorem} based on the arguments found in \cite{Kuksin} and \cite{Kuksin8} is given (the proof is summarized well in \cite{Shirikyan4}, p. 10). The main idea behind the argument is the following lemma (\cite{Kuksin7}, Lemma 3.2.6 or \cite{Shirikyan4}, Prop. 2.5), which can now depend on the $L^{\infty}$ norm of the force. 

Recall that a pair of random variables $\left(\zeta_{1},\zeta_{2}\right)$ defined on a probability space is called a {\it coupling} for the given measures $\mu_{1}, \mu_{2}$ if the distribution of $\zeta_{j}$ is $\mu_{j}$, $j=1,2.$

\begin{lemma}\label{thisisacouplinglemma}
Under the conditions of Theorem \ref{CTheorem}, there exists a constant $d>0$ such that for any points $u,u'\in B_{H}(R)$ with $\norm{u-u'}_{H}\leq d$ the measures $\beta(1,u_{1,2},\cdot)$ admit a coupling $V_{1,2}=V_{1,2}(u_{1},u_{2};\omega)$ such that 
\begin{equation}
\mathbb{P}\left\{\norm{V_{1}-V_{2}}_{H}\geq \frac{d}{2}\right\}\leq Cd \nonumber
\end{equation}
where $C>0$ does not depend on $u,u'.$
\end{lemma}
Since the conditions on the deterministic solution operator are imposed on each fixed time interval and the operator is the same for each interval, the kicked-equations have the same form as the time-independent and zero-force cases. Thus the proof of Lemma \ref{thisisacouplinglemma}, which depends on Condition \ref{Condition1a}, is nearly identical to the one in \cite{Kuksin7}, p. 118, except that constants now depend on the norm of the deterministic force $f$ (\cite{Kuksin7}, p. 117). 

It should be noted that the choice of $N$ in Theorem \ref{CTheorem} comes from the construction of the coupling in Lemma \ref{thisisacouplinglemma} and the construction only needs that it is sufficiently large. 

\begin{remark}
In \cite{Shirikyan4} the complication in proving Lemma \ref{thisisacouplinglemma} lies not in the choice of a time-periodic deterministic force, but the choice of random perturbation.
\end{remark} 

Given Lemma \ref{thisisacouplinglemma}, the remainder of the proof proceeds under the following two cases:
\begin{enumerate}
	\item If the initial conditions satisfy $\norm{u-u'}\leq d$ for $d$ small enough then Lemma \ref{thisisacouplinglemma} establishes that there is a positive probability that the random variables after one time step will be within $\frac{d}{2}$. By iteration there is a positive probability that the random variables will be within $\frac{d}{2^{n}}$ after $n$ time steps.
	\item If the initial conditions satisfy $\norm{u-u'}>d$, then by Condition \ref{controllability} there is a finite time $l$ where $\norm{u^{l}(u)-u^{l}(u')}\leq d$. After this, Lemma \ref{thisisacouplinglemma} again implies that the distance between the random variables is continually halved with positive probability. 
\end{enumerate}

The above argument gives the main idea behind the following lemma (\cite{Kuksin}, Lemma 3.3) : 

\begin{lemma} Let $u_{1,2}\in A$, where $A$ is the invariant set, and $d=\norm{u_{1}-u_{2}}$. Then under the condition of Theorem \ref{CTheorem} for any $k\geq 1$ the measures $\mu_{u_{1,2}}(k)$ admit a coupling $U_{1,2}^{k}= U_{1,2}^{k}(u_{1},u_{2},\omega^{k}), \ \omega^{k}\in\Omega^{k}$ such that
	\begin{enumerate}
		\item The maps $U_{1,2}^{k}$ are measurable with respect to $\left(u_{1},u_{2},\omega^{k}\right)\in A^{2}\times\Omega^{k}$.
		\item There exists a constant $\theta>0$ not depending on $u_{1}, \ u_{2},$ and $k$ such that
		\begin{equation}
			\mathbb{P}^{k}\left\{\norm{U_{1}^{k}-U_{2}^{k}}\leq d_{r}\right\}\geq \theta, \forall k\geq r+l(d_{0}), \ u_{1},u_{2}\in A.
		\end{equation}
		\item If $\norm{u_{1}-u_{2}}\leq d_{r}$ then
			\begin{equation}
				\mathbb{P}^{k}\left\{\norm{U_{1}^{k}-U_{2}^{k}}\leq d_{k+r}\right\}\geq 1-2^{-r-1}, \ k\geq 1, \ r\geq 0.    \label{singlecouplingestimate}
			\end{equation}
		\end{enumerate}
\end{lemma}
\noindent Due to Lemma \ref{thisisacouplinglemma}, the proof of this lemma is identical to the one in \cite{Kuksin} with the exception of using Condition \ref{controllability} here instead of Lemma 3.1 there.
The remainder of the proof follows identically to the argument in \cite{Kuksin8}.


Having established Theorem \ref{CTheorem}, it only remains to check that the conditions hold for the kicked Navier-Stokes equations. It is straightforward that Condition \ref{generalcontraction} implies Condition \ref{neededforabsorbingset}. Furthermore, since Conditions \ref{generalcontraction} and \ref{conditionforconvergence} are well known and analogous to results for bounded domains in $\R^{2}$ with smooth boundaries and periodic boundary conditions these are included in the appendix for completion. Instead only Condition \ref{controllability} is proved here. 

\subsection{Proof of Condition \ref{controllability}}

\noindent In order to establish Condition \ref{controllability} the following is needed (\cite{Kuksin2}, Lemma 5.4).

\begin{lemma}\label{arbitraryapproxofkicks} For any $\rho > 0$ and any integer $M\geq 1$, there is a $p_{0}=p_{0}(\rho,M)>0$ such that
    \begin{equation} \nonumber
        \mathbb{P}\left(\norm{\eta_{j}-x_{j}}_{H} <\rho, 1\leq j\leq M\right)\geq p_{0}
    \end{equation}
    uniformly in $x_{1},...,x_{M}\ in \ suppD(\eta)$ where $suppD(\eta)$ is the support of the distribution of the kicks.
\end{lemma}

\noindent The proof of Condition \eqref{controllability} uses the main idea behind Lemma 3.1 in \cite{Kuksin} and is split into the two cases considered.

\begin{lemma} Suppose that there exists an asymptotically stable solution, then for any $d>0$ and $R>0$ there exists integer $l=l(d,R)>0$ and real number $x=x(d)>0$ such that
        \begin{equation} \nonumber
 \mathbb{P}\left\{\|u^{l}(v_{0})-u^{l}(w_{0})\|_{H}\leq d\right\} \geq x, \ \text{for \ all} \ v_{0},w_{0}\in B_{H}(R).
        \end{equation}
\end{lemma}
\begin{proof} 

First fix all realization of the kicks as the zero realization. Then by assumption there exists a time $l$ such that
\begin{equation}
    \norm{u^{l}(w_{0})-u^{l}(v_{0})}_{H}\leq \dfrac{d}{2} \ \forall \ w_{0},v_{0}\in B_{H}(R).
\end{equation}
By continuity of the flow there is a $\gamma>0$ small enough that if $\norm{\eta_{k}}\leq \gamma$ for $1\leq k \leq l$ then
\begin{equation}
    \norm{u^{l}(w_{0})-u^{l}(v_{0})}_{H}\leq  d.
\end{equation}
By Lemma \ref{arbitraryapproxofkicks} the probability of $\norm{\eta_{k}}\leq \gamma$ is nonzero.
Thus
\begin{equation}
    \mathbb{P}\left\{\norm{u^{l}(w_{0})-u^{l}(v_{0})}_{H}\leq d, \ \forall w_{0},v_{0}\in B_{H}(R)\right\} \geq x.
\end{equation}
\end{proof}

Recall that the $N$ in Theorem \ref{CTheorem} is from the construction of the coupling in Lemma \ref{thisisacouplinglemma}. Let $N'$ be the maximum of the $N$ from the big kick assumption (and thus $\geq M$) and the $N$ generated by Lemma \ref{thisisacouplinglemma}.

\begin{lemma} Let $b_{j}>0$ for $1\leq j \leq N'$. Suppose that there exists a finitely stable point $u$ and assume that the big kick assumption holds, then for any $d>0$ and $R>0$ there exists an integer $l=l(d,R)>0$ and real number $x=x(d)>0$ such that
        \begin{equation} \nonumber
 \mathbb{P}\left\{\|u^{l}(v_{0})-u^{l}(w_{0})\|_{H}\leq d\right\} \geq x, \ \text{for \ all} \ v_{0},w_{0}\in B_{H}(R).
        \end{equation}
\end{lemma}
\begin{proof}
Let $\delta$ be the radius for the finitely stable point, $u$, and fix all realizations of the kicks as the zero realization. By (\ref{Hcontraction}) there exists a time $l$ such that  
\begin{equation}
	\norm{u^{l}(w)}_{H}\leq \dfrac{\delta}{4} +D(f)\ \forall \ w\in B_{H}(R).
\end{equation}
By the big kick assumption, there exists a kick $\eta'$ such that
\begin{equation}
	\norm{P_{M}u- P_{M}(u^{l}(w)+\eta')}_{H}\leq \dfrac{\delta}{2} \ \forall \ w\in B_{H}(R).
\end{equation}
Again fix all realizations as the the zero realization. By the assumption of a finitely stable point, there exists a time $k$ such that
\begin{equation}
	\norm{u^{l+k+1}(w)-u^{k}(u)}_{H}\leq \dfrac{d}{4}\ \forall \ w\in B_{H}(R).
\end{equation}
Thus there exists a time $k$ such that
\begin{equation}
	\norm{u^{l+k+1}(w_{0})-u^{l+k+1}(v_{0})}_{H}\leq \dfrac{d}{2}\ \forall \ w_{0},v_{0}\in B_{H}(R).
\end{equation}
By continuity, $\gamma>0$ can be chosen such that if $\norm{\eta_{j}}\leq \gamma$ for $1\leq j \leq l$, $\norm{\eta'-\zeta}\leq \gamma$ where $\zeta$ is another realization of the kick, and $\norm{\eta_{j}}\leq \gamma$ for $l+1< j\leq l+k+1$ then

\begin{equation}
		\norm{u^{l+k+1}(w_{0})-u^{l+k+1}(v_{0})}_{H}\leq d\ \forall \ w_{0},v_{0}\in B_{H}(R).
\end{equation}
By Lemma \ref{arbitraryapproxofkicks} there is a positive probability of the kicks satisfying the inequalities. 
\end{proof}

\noindent This completes the proof of Condition \ref{controllability} and thus there is uniqueness of invariant measure in $H$. 

\section{Support of the Measure}

Before stating the main result of this section, we recall some definitions and straightforward results about the support of a measure.

\begin{definition} The support of a measure $\mu$ on H is the smallest closed subset $K$ in $H$ such that $\mu(H/K)=0$. A measure is concentrated on a set $B$ if $\mu(B)= 1.$
\end{definition}

\noindent To continue we need the following results from \cite{Kuksin2}, Lemmas 5.4 and 5.5. The first is a restatement of Lemma \ref{arbitraryapproxofkicks}.

\begin{lemma}\label{arbitraryapproxofkicks2} A sequence of realizations of kicks can be taken arbitrarily close to any prescribed sequence of vectors in $supp D(\eta)$ with positive probability. 
\end{lemma} 

\begin{definition} For $y\in H$, let $A_{0}(y)= y$. Let $A_{n}(y)= S(A_{n-1}(y))+supp D(\eta_{k})$, called the set of attainability from the set $\left\{y\right\}$ at time n. The set of attainability from $y$ is defined as
\begin{equation} \nonumber
	A(y) = \overline{\bigcup_{i=0}^{\infty}A_{i}(y)}.  \label{setofattainability}
\end{equation}
\end{definition}

\begin{lemma}\label{approxinfinitetime} For any $r>0$ there is an integer $k\geq 0$ such that $A(y)$ is contained in the r-neighborhood of $A_{k}(y)$, i.e. for any $a\in A(y)$ there exists $a_{k}\in A_{k}(y)$ such that $a_{k}\in B_{H}(r,a)$, where $B_{H}(r,a)$ is the ball of radius $r$ in $H$ centered at $a$. 
\end{lemma}

\noindent The definition of the set of attainability is similar to Condition \ref{neededforabsorbingset} except that the ball is centered at $y$ instead of 0. 

\begin{remark}
The support of the measure for the Navier-Stokes equations is concentrated on $V$ (\cite{Varner}, Lemma 5.5.2) and, in general, the support of the measure is contained in a ball centered around the origin of radius the square root of
\begin{equation}
	\left(\dfrac{\norm{f}_{L^{\infty}(0,\infty;H)}^{2}}{\nu^{2}\lambda_{1}}+B_{0}\right)
	\dfrac{1}{1-e^{-\lambda_{1}\nu}},   \label{sizeofabsorbingball}
\end{equation}
where $\norm{\eta_{k}}^{2}\leq B_{0}$ for all k (\cite{Varner}, Lemma 5.5.3). 

\noindent When there is an asymptotically stable solution the support is contained in a ball of radius
\begin{equation}
	\frac{\sqrt{B_{0}}}{1-e^{-L}}
\end{equation}
centered at the asymptotically attracting solution (\cite{Varner}, Lemma 5.2.1), where $L$ is the rate of convergence ($q^{t}=e^{-L}$ for $0<q<1$).
\end{remark}

\subsection{Support of the Measure} 

We next extend the standard definitions of wandering and nonwandering points (\cite{Dymnikov2}, page 27) to the case of stochastic flow.

\begin{definition} Let $U_{\epsilon}(p) = B_{H}(\epsilon,p)-\left\{p\right\}.$
  A point $p$ in H is nonwandering if for all $\epsilon>0$ and for every $T>0$ there exists $t>T$ such that
  \begin{equation}\nonumber
  	B_{H}(\epsilon,p)\cap S_{t}(U_{\epsilon}(p)) \neq \oslash.
  \end{equation}
\end{definition}

\begin{definition} 
  A point $p$ in H is wandering if there exists $\epsilon>0$ and there exists $T>0$ such that for all $t>T$
  \begin{equation}\nonumber
  	B_{H}(\epsilon,p)\cap S_{t}(U_{\epsilon}(p)) = \oslash.  \label{defwandering}
  \end{equation}
\end{definition}

A point is defined as wandering or nonwandering based on the behavior of nearby points. One consequence of this is that an unstable point will behave like a wandering point (in the sense of \cite{Dymnikov2}, page 27). For example, if the force is stationary, an unstable stationary solution would now behave as a wandering point.

The following result was proved in \cite{Varner}, Theorem 5.5.8. 
\begin{theorem}
	Let $A$ be the set of attainability from the the set of nonwandering points. Then any $a\in A$ is in the support of the measure.
\end{theorem}

\noindent We outline the proof, which is similar to the steps in \cite{Kuksin2}, p. 320.

\begin{itemize}
	\item By time-invariance, for any $r,\epsilon>0$ and any $v\in H$  
\begin{equation}
	\begin{split}
			\mu(B_{H}(r,a)) = \int_{H}\beta(l,v,B_{H}(r,a))\mu(dv)
			& \geq \int_{B_{H}(\epsilon,v)}\beta(l,u,B_{H}(r,a))\mu(du) 
	\end{split}
\end{equation}

\item It is necessary to show that for any $a\in A$ there exists $v\in H$, $\epsilon>0$, and times $t_{1}$, $t_{2}$ such that
\begin{equation}
	\beta(t_{1}+t_{2},u,B_{H}(r,a))>0, \ \forall \ u\in B_{H}(\epsilon,v).
\end{equation}

\item By the definition of the set of attainability, there is a nonwandering point $y$ such that $a$ is accessible from $y$. Since for any $u\in B_{H}(\epsilon,v)$
\begin{equation}
		\begin{split}
			 \beta\left(t_{1}+t_{2},u,B_{H}(r,a)\right)
			& = \int_{H}
				\beta\left(t_{1},u,dw\right)\beta\left(t_{2},w,B_{H}(r,a)\right) \\
			& \geq \int_{B_{H}\left(\delta,y\right)}
				\beta\left(t_{1},u,dw\right)\beta\left(t_{2},w,B_{H}(r,a)\right) \\
			&	\geq \inf_{w\in B_{H}\left(\delta,y\right)}\beta\left(t_{2},w,B_{H}(r,a)\right)
							\beta\left(t_{1},u,B_{H}\left(\delta,y\right)\right)
	\end{split}
\end{equation}
it is enough that $\inf_{w\in B_{H}\left(\delta,y\right)}\beta\left(t_{2},w,B_{H}(r,a)\right)$ and $\beta\left(t_{1},u,B_{H}\left(\delta,y\right)\right)$ are strictly positive, which follows from the next two lemmas. 
\end{itemize}

\begin{lemma} Let $y$ be a nonwandering point. For any $\delta>0$ there is $\epsilon>0$, an integer $t_{1}= t_{1}(\delta)\geq 0$, and a constant $x=x(\delta)> 0$ such that
\begin{equation} \nonumber
	\mathbb{P}\left\{\norm{y-u^{t_{1}}(w)}_{H} \leq \delta \ \forall w\in B_{H}(\epsilon,y)\right\} 
	\geq x.   \label{positiveprobfornonwander} 
\end{equation}
\end{lemma}

\noindent The proof is very similar to that of Condition \ref{controllability} and thus only a sketch is given. By the definition of a nonwandering point, the intersection of any open ball (for example of radius $\delta/4$) around a nonwandering point, $y$, has a non-empty intersection with the deterministic flow of the set at some time $t_{1}$. By continuity, there exists $\epsilon>0$ such that if the initial condition $w \in B_{H}(\epsilon,y)$ and the kicks are small enough then $u^{t_{1}}(w)\in B_{H}(\delta,y)$ (see \cite{Varner}, Lemma 5.5.11 or \cite{Kuksin2}, p. 321-322 ).

\begin{lemma} Let $A$ be the set of attainability from the set of nonwandering points. For any $a\in A$ and any $r>0$ there exists $\delta>0$ and a nonwandering point $y$ such that for some time $t_{2}$ and all $v\in B_{H}(\delta,y)$,
\begin{equation} \nonumber
	\beta(t_{2},v,B_{H}(r,a)) > 0.  \label{intoaball}
\end{equation} 
\end{lemma}
\noindent The proof is nearly a repeat of the argument made in \cite{Kuksin2}, page 322, with modifications for the change in the definition of the set of attainability, so only a brief sketch is given. By Lemma \ref{approxinfinitetime}, there is $a_{k}\in B_{H}(r/2,a)$ such that $a_{k}$ is attainable by a finite sequence of fixed kicks from the nonwandering point $y$. By continuity of the flow and the properties of the kicks, there is a $\delta>0$ and $\gamma>0$ such that if $v\in B_{H}(\delta,y)$ and the kicks vary by at most $\gamma$ then there is a positive probability that $u^{t_{2}}(v)\in B_{H}(r,a)$.

Due to the existence of an asymptotically stable solution when the force is small enough, gives a zonal solution, or gives an almost zonal solution, the following holds.

\begin{corollary} If the force
\begin{enumerate}
	\item is small enough -see Remark \ref{differenceofsolutionssatisfiescontraction}
\item yields a zonal solution
\item yields an almost zonal solution
\end{enumerate}
then the support of the measure is the set of attainability from the unique exponentially stable periodic solution.
\end{corollary}

\section{Conclusion}

While there is invariant measure for the kicked Navier-Stokes equations with a bounded time-periodic deterministic force, it is only possible to give a clear description of the support of the measure in a few limited situations. Furthermore, since the kicks can be taken arbitrarily small (with the first $N$ nonzero), if there is an asymptotically stable solution then the support can be considered to nearly be the unique globally stable periodic solution. Unfortunately, for more general forces the support of the measure is not as clear. For example, it is not as clear what nonwandering points may exist. While the assumption of finitely stable point gives that there is a (at least one) periodic solution (possibly with the same period as the force), the size requirement on the kick is problematic both for understanding the support of the measure and for meteorological considerations. 

It is possible, however, that the kicks may be allowed to be smaller. The big kick assumption is to ensure that a kick can, with positive probability, send the flow into the neighborhood of any point in the deterministic absorbing ball. The reason for the big kick assumption comes from the deterministic setting where a dirac measure at any stationary solution is a time-invariant measure, giving non uniqueness if there are multiple stationary solutions. For example, if there are two stable stationary solutions the kicks must be (at minimum) large enough to send the flow from inside the radius of stability of one into the radius of stability of the other. The big kick assumption is sufficient to do this, but a smaller kick may suffice.

Of course, the majority of the results presented in this paper apply to the Navier-Stokes equations on the torus. While the results concerning the zonal solutions no longer hold, if the force yields a unique asymptotically stable solution then the support of the measure is again straightforward to describe.

\section{Appendix}


\subsection{Estimates}

We now present estimates that will be needed to establish Conditions \ref{generalcontraction} and \ref{conditionforconvergence} and Lemmas \ref{thiszonallemma} and \ref{AlmostZonalProp}. With the exception of Lemma \ref{lemmaforzonalestimates} these estimates are analogous to standard estimates on flat-domains with periodic boundary conditions.

\begin{lemma} For $u\in V$ the Poincare Inequality holds, i.e.
\begin{equation}
	\norm{u}_{V}^{2}=
	\|A^{1/2}u\|_{H}^{2} 
	\geq \lambda_{1}\norm{u}_{H}^{2} \label{Poincare}
\end{equation}
where $\lambda_{1}$ is the first eigenvalue of the Laplacian.
In particular, the V-norm is equivalent to the $H^{1}$-norm on $V$.
\end{lemma}
\noindent The proof is identical to the case of flat domains due to the existence of an orthonormal basis. Furthermore, $\lambda_{1}$ is the first eigenvalue of the scalar Laplacian on the sphere (\cite{Ilin}, p.567).


\begin{lemma}\label{trilinearestimateslemma} For $u,v,w \in \ V$, the trilinear form satisfies
\begin{align}
	& b(u,v,v)= 0, \ \ b(u,v,w)= - b(u,w,v),    \label{bproperties} \\
	& |b(u,v,w)|\leq k\norm{u}_{H^{1}}\norm{v}_{H^{1}}\norm{w}_{H^{1}},   \label{bestimate}\\ 
	& \left|b(u,v,w)\right|
	\leq
		k\norm{u}^{1/2}_{H}\norm{u}^{1/2}_{H^{1}}
		\norm{v}_{H^{1}}\norm{w}^{1/2}_{H}
		\norm{w}^{1/2}_{H^{1}}. \label{bestimate2}
\end{align}
If $v\in H^{2}\cap V$ then
\begin{align}
	&  b(v,v,Av)=0,     \label{bproperty} \\
	&	|b(u,v,w)|\leq k\norm{w}_{H}\norm{v}_{H^{2}}\norm{u}_{H^{1}}, \label{bestimate3} \\
	& \left|b(u,v,w)\right|
	\leq
		k\norm{u}_{H}^{1/2}\norm{u}_{H^{1}}^{1/2}
		\norm{v}^{1/2}_{H^{1}}
		\norm{v}^{1/2}_{H^{2}}\norm{w}_{H}.  \label{bestimate4}.  
\end{align}  

Furthermore, let $u\in H^{2}\cap V$ be a zonal vector field (only latitudinal dependence) and $v\in H^{2}\cap V$ then 
\begin{align}
	&  b(u,v,Av)=0,  \label{strongerthanskiba} \\
	&	 b(v,u,Av) =0. \label{Calc3}
\end{align}                                      
\end{lemma}

\begin{proof}
Since the proof of \eqref{bproperties} and \eqref{bproperty} are identical to the ones in \cite{Ilin}, pp. 566-568, the proofs of \eqref{bestimate}, \eqref{bestimate2}, \eqref{bestimate3}, and \eqref{bestimate4} follow from applying the H\"older inequality and the Ladyzhenskya inequality after taking extensions (\cite{Ilin}, pp. 566-567) and the proof of \eqref{strongerthanskiba} is identical to the calculation on p. 69 of \cite{Ilin2} (which uses Lemma 4.4 on p. 62 there), we will only prove \eqref{Calc3} here.

Now recall that (\cite{Ilin}, pp. 567-568) since the sphere is simply connected, for a divergence-free vector field $u$, there is a flow function $\psi$
\begin{equation} \nonumber
		 u= -\curl \psi\vec{n} = \vec{n}\times \nabla \psi, \quad
		\curl_{n}u = \Delta \psi \vec{n}.
\end{equation} 

\noindent Furthermore, the standard spherical Jacobian will be needed and is given by (\cite{Ilin2} p.51)
$$J(a,b)=-\curl_{n}\left(\vec{n}\times \left(\vec{n}\times\nabla b\right)\right) = \frac{1}{\cos\phi}\left(\frac{\partial a}{\partial \lambda}\frac{\partial b}{\partial \phi}-\frac{\partial b}{\partial \lambda}\frac{\partial a}{\partial \phi}\right).$$

\noindent Note that by Stoke's Theorem $\int_{M}J(a,b)dM=0$.

The proof of equation \eqref{Calc3} uses an argument similar to \cite{Ilin2}, 70. Recall that $B(v,u)= \curl_{n}v\times u$ and denote the flow functions for $u$ and $v$ as $\bar{\psi}$ and $\psi$, respectively. 
\begin{equation}
	\begin{split}
		 \left\langle \curl_{n}v\times u,Av \right\rangle
		&	= 
			\left\langle \curl_{n}\left(\curl_{n}v\times u\right),\curl_{n}v \right\rangle \\
		& = \left\langle \curl_{n}\left(\vec{n}\Delta\psi\times u\right),\Delta\psi \right\rangle
			= \left\langle J(\bar{\psi},\Delta\psi),\Delta\psi \right\rangle \\
		& = -\left\langle \dfrac{1}{\cos\phi}\partial_{\phi}\bar{\psi}\partial_{\lambda}\Delta\psi, \Delta\psi\right\rangle \\
		& = -\dfrac{1}{2}\int_{M} \dfrac{1}{\cos\phi}\partial_{\phi}\bar{\psi}\partial_{\lambda}\left[\Delta\psi\right]^{2}dM \\
		& = -\dfrac{1}{2}\int_{M}J(\bar{\psi},\left[\Delta\psi\right]^{2})dM=
		0 . 
	\end{split}
\end{equation}

\noindent Note that $\Delta$ is the standard spherical Laplacian on the sphere.
\end{proof}


The following lemma will allow for the Coriolis term $C(u)$ to vanish from all the estimates. Its proof only uses that the Laplacian commutes with differentiability in the longitudinal direction - see \cite{Skiba1}, p. 635.

\begin{lemma}\label{lemmaforzonalestimates} 
For smooth vector fields $u$, the following holds for $r\geq 0$
\begin{equation}
	\left\langle C(u),v\right\rangle_{H}
	=\left\langle l\,\vec{n}\times u, A^{r}u\right\rangle= 0.  \label{Skiba-estimate}
\end{equation}
\end{lemma}

\noindent We now turn to the proofs of Conditions \ref{generalcontraction} and \ref{conditionforconvergence} and Lemmas \ref{thiszonallemma} and \ref{AlmostZonalProp}. Since many of the calculations are standard, only the main steps are given. Recall that $f\in L^{\infty}(0,\infty;H)$.

\subsection{Proof of Condition \ref{generalcontraction}}

Let $S_{t}u_{0}$ be the solution of the 2D Navier-Stokes equations with initial condition $u_{0}$ at time $t$.


\begin{lemma}
The following inequalities hold for the deterministic 2D Navier-Stokes equation on the sphere for all $t \geq 0$:
 \begin{align}
         &    \norm{S_{t}u_{0}}^{2}_{H} \leq \norm{u_{0}}^{2}_{H}e^{-\lambda_{1}\nu t}
             +
             \dfrac{\norm{f}_{L^{\infty}(0,\infty;H)}^{2}}{\nu^{2}\lambda_{1}}
              \left(1-e^{-\nu\lambda_{1}t}\right)
              \label{Hcontraction}\\
         & \norm{S_{t}u_{0}}^{2}_{H^{1}}
             \leq \norm{u_{0}}^{2}_{H^{1}}e^{-\lambda_{1}\nu t}
             +
             \dfrac{\norm{f}^{2}_{L^{\infty}(0,\infty;H)}}{\nu^{2}\lambda_{1}}
             \left(1-e^{-\nu\lambda_{1}t}\right),
             \label{Vcontraction}
     \end{align}
     where $\lambda_{1}$ is the first eigenvalue of the operator $-\Delta$ on functions.

Moreover, for any $t\geq \dfrac{1}{2}$
\begin{equation}
    \norm{S_{t}u_{0}}^{2}_{H^{1}} \leq K\norm{u_{0}}^{2}_{H}e^{-\nu\lambda_{1}t}+ C_{1}\norm{f}_{L^{\infty}(0,\infty;H)}^{2}. \label{VtoHcontraction}
\end{equation}
\end{lemma}

\begin{proof}
The proof follows the estimates in \cite{Ilin}, p. 572. Take the $L^{2}$ inner product of the Navier-Stokes equation with $u$. By \eqref{Skiba-estimate} and \eqref{bproperties}
\begin{align}
    & \dfrac{1}{2}\partial_{t}\norm{u}_{H}^{2}
        +
        \nu\|A^{1/2}u\|_{H}^{2}
        =
        \left\langle f,u\right\rangle  \nonumber\\
     \Rightarrow \
    & \dfrac{1}{2}\partial_{t}\norm{u}_{H}^{2}
        +
        \nu\norm{u}_{H^{1}}^{2}
        \leq
        \dfrac{1}{2\nu}\norm{f}_{L^{\infty}(0,\infty;H)}^{2}
        +
        \dfrac{\nu}{2}\norm{u}_{H^{1}}^{2}\label{tointegrate}   \\
     \Rightarrow \
    &    \partial_{t}\norm{u}_{H}^{2}
        \leq
        -\nu\lambda_{1}\norm{u}_{H}^{2} + \dfrac{1}{\nu}\norm{f}_{L^{\infty}(0,\infty;H)}^{2}
        \quad \text{by \ (\ref{Poincare})}.        \nonumber
\end{align}
This gives
\begin{equation}
    \norm{S_{t}u_{0}}^{2}_{H} \leq \norm{u_{0}}^{2}_{H}e^{-\lambda_{1}\nu t}
    +
 \dfrac{1}{\lambda_{1}\nu^{2}}\norm{f}_{L^{\infty}(0,\infty;H)}^{2}\left(1-e^{\lambda_{1}\nu t}\right),
\end{equation}
establishing \eqref{Hcontraction}.

For \eqref{Vcontraction}, take the $L^{2}$ inner product with $Au$. By \eqref{Skiba-estimate} and \eqref{bproperty}
\begin{align}
    & \dfrac{1}{2}\partial_{t}\norm{u}_{H^{1}}^{2}
        +
        \nu\norm{u}_{H^{2}}^{2}
        +
        b(u,u,Au)
        =
        \left\langle f,u\right\rangle  \nonumber \\
     \Rightarrow \
    &    \dfrac{1}{2}\partial_{t}\norm{u}_{H^{1}}^{2}
        +
        \nu\norm{u}_{H^{2}}^{2}
        \leq
        \dfrac{1}{2\nu}\norm{f}_{L^{\infty}(0,\infty;H)}^{2}
        +
        \dfrac{\nu}{2}\norm{u}_{H^{2}}^{2}\label{tointegrateforH2} \quad \text{by \ \eqref{bproperty}} \\
     \Rightarrow \
    &    \partial_{t}\norm{u}_{H^{1}}^{2}
        \leq
        -\nu\lambda_{1}\norm{u}_{H^{1}}^{2}
        +
        \dfrac{1}{\nu}\norm{f}_{L^{\infty}(0,\infty;H)}^{2} \quad \text{by \ (\ref{Poincare})}.\nonumber
\end{align}
Therefore
\begin{equation}
    \norm{S_{t}u_{0}}^{2}_{H^{1}}
    \leq \norm{u_{0}}^{2}_{H^{1}}e^{-\lambda_{1}\nu t}
        +
 \dfrac{1}{\lambda_{1}\nu^{2}}\norm{f}_{L^{\infty}(0,\infty;H)}^{2}\left(1-e^{\lambda_{1}\nu t}\right),
\end{equation}
establishing \eqref{Vcontraction}.

For \eqref{VtoHcontraction}, note that integrating \eqref{tointegrate} from $t_{0}$ to $t+t_{0}$ gives
\begin{align}
    & \norm{S_{t+t_{0}}u_{0}}_{H}^{2}
        -
        \norm{S_{t_{0}}u_{0}}_{H}^{2}
        +
 \nu\int_{t_{0}}^{t+t_{0}}\norm{S_{\tau}u_{0}}_{H^{1}}^{2}d\tau
        = \dfrac{t}{\nu}\norm{f}_{L^{\infty}(0,\infty;H)}^{2} \nonumber \\
     \Rightarrow \
    &  \int_{t_{0}}^{t+t_{0}}\norm{S_{\tau}u_{0}}_{H^{1}}^{2}d\tau
        \leq
        \dfrac{1}{\nu}\norm{S_{t_{0}}u_{0}}_{H}^{2}
        +
        \dfrac{t}{\nu^{2}}\norm{f}_{L^{\infty}(0,\infty;H)}^{2}. \label{integralinequality}
\end{align}
\eqref{Vcontraction} implies that for any $t\geq \dfrac{1}{2}$ and any $t-\dfrac{1}{2}<t_{0}<t$
\begin{equation}
    \norm{S_{t}u_{0}}^{2}_{H^{1}} \leq \norm{S_{t_{0}}u_{0}}^{2}_{H^{1}}+
    \dfrac{\norm{f}_{L^{\infty}(0,\infty;H)}^{2}}{\nu^{2}\lambda_{1}} \label{tointegrate2}.
\end{equation}
Integrating \eqref{tointegrate2} with respect to $t_{0}$ from $t-\dfrac{1}{2}$ to $t$ gives, using \eqref{integralinequality} and \eqref{Hcontraction}
\begin{equation}
\begin{split}
     \dfrac{1}{2}\norm{S_{t}u_{0}}^{2}_{H^{1}}
    &    \leq
        \int_{t-(1/2)}^{t}\norm{S_{t_{0}}u_{0}}^{2}_{H^{1}}dt_{0}
        +
        \dfrac{\norm{f}_{L^{\infty}(0,\infty;H)}^{2}}{2\nu^{2}\lambda_{1}} \\
    & \leq
        \dfrac{1}{\nu}\norm{S_{t-(1/2)}u_{0}}_{H}^{2}
        +
        \dfrac{1}{2\nu^{2}}\norm{f}_{L^{\infty}(0,\infty;H)}^{2}
        +
        \dfrac{\norm{f}_{L^{\infty}(0,\infty;H)}^{2}}{2\nu^{2}\lambda_{1}} \\
    & \leq
        C\norm{u_{0}}_{H}^{2}e^{-\nu\lambda_{1}t} + C_{1}\norm{f}_{L^{\infty}(0,\infty;H)}^{2},
\end{split}
\end{equation}
establishing \eqref{VtoHcontraction}.
\end{proof}

\noindent Now consider the difference between two solutions $w=u-v$
    \begin{equation}
         S_{t}u-S_{t}v =
            \dfrac{\partial w}{\partial t} +
            \nu A+B(w,u)+ B(v,w) + C(w)=0. \label{differenceofsolutions}
    \end{equation}


\begin{lemma} For any $R>0$ and for all $t\geq 0$ the difference of solutions satisfies
\begin{equation}
    \norm{S_{t}w_{0}}_{H}^{2}\leq C(R,f,t)\norm{w_{0}}_{H}^{2}, \label{continuityofNS}
\end{equation}
whenever $\norm{u_{0}}_{H}\leq R$ and $\norm{v_{0}}_{H}\leq R$.
\end{lemma}

\begin{proof}
Taking the $L^{2}$ inner product with $w$
\begin{equation}
         \dfrac{1}{2}\partial_{t}\norm{w}^{2}_{H}
             +
             \nu \norm{w}^{2}_{H^{1}}
             +
             b(w,u,w)
             +
             b(v,w,w)
             = 0.
\end{equation}
Therefore, by \eqref{bproperties},
\begin{align}
         \dfrac{1}{2}\partial_{t}\norm{w}^{2}_{H}
             +
             \nu \norm{w}^{2}_{H^{1}}
        &    \leq |b(w,u,w)|  \\
        &    \leq k\norm{w}_{H^{1}}\norm{u}_{H^{1}}\norm{w}_{H}\quad \text{by \ \eqref{bestimate2}}.  \nonumber 
\end{align}
By the Cauchy inequality
\begin{equation}
            \partial_{t}\norm{w}^{2}_{H}
            +
            \nu\norm{w}^{2}_{H^{1}}
            \leq
            \dfrac{k^{2}}{\nu}\norm{w}^{2}_{H}\norm{u}^{2}_{H^{1}}
             \label{forVnorm}
\end{equation}
and thus by \eqref{Poincare}
\begin{equation}
   \norm{S_{t}w_{0}}_{H}^{2}
            \leq
            \text{exp}\left(-\nu\lambda_{1}t+
            \displaystyle\int_{0}^{t}\dfrac{k^{2}}{\nu}
 \norm{S_{s}u_{0}}^{2}_{H^{1}}ds\right)\norm{w_{0}}_{H}^{2}.
            \label{cansimplifyforstationary}
\end{equation}
By \eqref{integralinequality}
\begin{equation}
    \norm{S_{t}w_{0}}_{H}^{2}
    \leq
    \text{exp}\left(-\nu\lambda_{1}t+
            \dfrac{k^{2}}{\nu^{2}}
            \norm{u_{0}}^{2}_{H}
            +
 \dfrac{k^{2}}{\nu^{3}}\norm{f}_{L^{\infty}(0,\infty;H)}^{2}t\right)\norm{w_{0}}_{H}^{2}.
            \label{differenceofsolutions2}
\end{equation}
The exponential is less than or equal to some constant (depending on R and the norms of $f$) for any fixed $t\geq 0$.
\end{proof}

\begin{remark}\label{differenceofsolutionssatisfiescontraction} By \eqref{differenceofsolutions2} in order to ensure \eqref{FForceCondition} it is sufficient that
\begin{equation} \nonumber
    \norm{f}_{L^{\infty}(0,\infty;H)} < \dfrac{\nu^{2}\sqrt{\lambda_{1}}}{k}. \label{thisishowsmall}
\end{equation}

\noindent If equation \eqref{thisishowsmall} is satisfied, there is a unique globally exponentially stable solution that is periodic with the same period as the force.
\end{remark}

\subsection{Proof of Condition \ref{conditionforconvergence}}


\begin{lemma} Let $w=u-v$ and for any $R>0$ let $\norm{u_{0}}_{H}<R$ and $\norm{v_{0}}_{H}<R$. The following estimate holds for all $t\geq 1$:
\begin{equation}
    \norm{S_{t}w_{0}}_{H^{1}} \leq
    C(R,\norm{f}_{L^{\infty}(0,\infty;H)},t)\ \norm{w_{0}}_{H} . \label{HboundonVdifference}
\end{equation}
\end{lemma}
\begin{proof}
  Integrating \eqref{tointegrateforH2} from $s$ to $t$ gives
  \begin{align}
       \nu\int_{s}^{t}\norm{S_{\tau}u_{0}}_{H^{2}}^{2}d\tau 
    = & \norm{S_{s}u_{0}}_{H^{1}}^{2} - \norm{S_{t}u_{0}}_{H^{1}}^{2}
      + \dfrac{t}{\nu}\norm{f}_{L^{\infty}(0,\infty;H)}^{2} \nonumber \\
 \leq &
      \norm{S_{s}u_{0}}_{H^{1}}^{2}+ \dfrac{t}{\nu}\norm{f}_{L^{\infty}(0,\infty;H)}^{2}. \label{mvtestimate}
  \end{align}
Using \eqref{VtoHcontraction} gives for $1/2<s<t$
\begin{equation}
    \nu\int_{s}^{t}\norm{S_{\tau}u_{0}}_{H^{2}}^{2}d\tau
    \leq
    C\norm{u_{0}}_{H}^{2}+ tC(\norm{f}_{L^{\infty}(0,\infty;H)}) +C(\norm{f}_{L^{\infty}(0,\infty;H)}). \label{H2integralbound}
\end{equation}
Integrating \eqref{forVnorm} from 1/2 to 1 and by the Mean Value Theorem there is $s\in \left(1/2,1\right)$ such  that
\begin{align}
        \nu\norm{S_{s}w_{0}}_{H^{1}}^{2} \nonumber
    &    =
        2\nu\int_{1/2}^{1}\norm{S_{\tau}w_{0}}_{H^{1}}^{2}d\tau \quad \text{and \ by \ \eqref{mvtestimate}}  \nonumber \\
    &    \leq
        C(R,1/2)\norm{w_{0}}_{H}^{2}
        +
        \dfrac{k^{2}}{\nu}\int_{1/2}^{1}\norm{S_{\tau}w_{0}}_{H}^{2}
        \norm{S_{\tau}u_{0}}_{H^{1}}^{2}d\tau \nonumber \\
    &    \leq
        C(R,\norm{f}_{L^{\infty}(0,\infty;H)})\norm{w_{0}}_{H}^{2} \quad \text{by \ \eqref{integralinequality} \ and \ \eqref{continuityofNS}}.
        \label{boundfors}
\end{align}
Taking the $L^{2}$ inner product of \eqref{differenceofsolutions} with $Aw$ gives
\begin{equation}
     \dfrac{1}{2}\partial_{t}\norm{w}_{H^{1}}^{2}
     +
     \nu\norm{w}_{H^{2}}^{2}
        \leq \left|b(u,w,Aw)\right| + \left|b(w,v,Aw)\right|.
        \label{boundedinnextstep}
\end{equation}
By \eqref{bestimate4} and \eqref{bestimate3} respectively, the right side of \eqref{boundedinnextstep} is bounded above by
\begin{equation}
\begin{split}
    & \leq k\norm{u}_{H^{1}}^{1/2}\norm{u}_{H}^{1/2}\norm{w}_{H^{1}}^{1/2}
        \norm{w}_{H^{2}}^{1/2}\norm{w}_{H^{2}}
        +  k\norm{w}_{H^{1}}\norm{v}_{H^{2}}\norm{w}_{H^{2}} \\
    & \leq
        K\norm{u}^{2}_{H^{1}}\norm{u}_{H}^{2}\norm{w}^{2}_{H^{1}}
        + \dfrac{\nu}{2}\norm{w}_{H^{2}}^{2}
        + C\norm{w}_{H^{1}}^{2}\norm{v}_{H^{2}}^{2}\quad    \text{by \ Cauchy}.
\end{split}
\end{equation}
Therefore
\begin{equation}
     \partial_{t}\norm{w}_{H^{1}}^{2}
        \leq
        \left(-\nu\lambda_{1}+K\left(\norm{u}_{H^{1}}^{2}+
        \norm{v}_{H^{2}}^{2}\right)\right)\norm{w}_{H^{1}}^{2}
\end{equation}
and
\begin{eqnarray}
    & \norm{S_{t}w_{0}}_{H^{1}}^{2} \leq \norm{S_{s}w_{0}}_{H^{1}}^{2}\times \nonumber \\
    & \text{exp}\left(-\nu\lambda_{1}(t-s) + k\int_{s}^{t}
 \left(\norm{S_{\tau}u_{0}}_{H^{1}}^{2}\norm{S_{\tau}u_{0}}_{H}^{2}
        +\norm{S_{\tau}v_{0}}_{H^{2}}^{2}\right)d\tau\right).
\end{eqnarray}
By \eqref{Hcontraction} and \eqref{H2integralbound} this is bounded above by
\begin{eqnarray}
        \leq & 
            \norm{S_{s}w_{0}}_{H^{1}}^{2}\text{exp}\left(-\nu\lambda_{1}(t-1)\right)\times \nonumber \\
        &    \text{exp}\left(C(R,\norm{f}_{L^{\infty}(0,\infty;H)})\int_{0}^{t}
            \norm{S_{\tau}u_{0}}_{H^{1}}^{2}d\tau\right)\times \nonumber \\
        &    \text{exp}\left(C(R)+tC\norm{f}_{L^{\infty}(0,\infty;H)}^{2}\right).
\end{eqnarray}
By \eqref{integralinequality} and \eqref{boundfors} this is bounded above by
\begin{eqnarray}
  \leq & 
         C(R,\norm{f}_{L^{\infty}(0,\infty;H)})\norm{w_{0}}_{H^{1}}^{2} \text{exp}\left(-\nu\lambda_{1}(t-1)\right)
         \times \nonumber \\
       &     \text{exp}\left(C(R,\norm{f}_{L^{\infty}(0,\infty;H)})
            +C(R)+tC\norm{f}_{L^{\infty}(0,\infty;H)}^{2}\right).
\end{eqnarray}
This establishes \eqref{HboundonVdifference}.
\end{proof}

\noindent Let $Q= (I-P_{N})S_{t}w_{0} = (I-P_{N})(S_{t}(u_{0}-v_{0}))$.

\noindent Then
\begin{eqnarray}
        &    \norm{Q}_{H}^{2} 
            \leq \dfrac{1}{\lambda_{N+1}}\norm{Q}_{H^{1}}^{2}
            \leq \dfrac{1}{\lambda_{N+1}}\norm{S_{t}w_{0}}_{H^{1}}^{2} \nonumber \\
        &    \leq \dfrac{C(R,\norm{f}_{L^{\infty}(0,\infty;H)},t)}{\lambda_{N+1}}\norm{w_{0}}_{H}^{2}
            :=\gamma_{N}\norm{w_{0}}_{H}^{2},
\end{eqnarray}
where the last step is by \eqref{HboundonVdifference}. For any $t\geq 1$ a $N$ can be found (depending on t, R, and f) such that the $\gamma_{N}$ is less than or equal to any $q>0$. Since $t=1$ for the kicked equations, $N$ can be chosen only depending on $R$ and $f$.

\subsection{Proof of Lemma \ref{thiszonallemma}}

The proof is analogous to a calculation in \cite{Ilin2}, pp. 69-70 (done for $f=2\nu\curl(-a\sin\phi)$), and the steps given in \cite{Varner}, Lemma 5.3.2, where a general time-independent zonal force is considered.

\noindent Let $u=\bar{u}+u'$ solve the time-dependent Navier-Stokes equations with forcing $f$ where $u'$ is a perturbation and $\bar{u}$ is the zonal solution. 
The perturbation solves
\begin{equation}
	\partial_{t}u'+\nu Au' + Gu' + B(u',u')=0,
\end{equation}
where 
\begin{equation}
	Gu'= C(u') + B(\bar{u},u')+B(u',\bar{u}).
\end{equation}
Dropping the primes for ease of notation and taking the inner product with $Au$ gives
\begin{equation}
	\dfrac{1}{2}\partial_{t}\norm{u}_{1}^{2}
	+
	\nu \norm{u}_{2}^{2}
	+ 
	\left\langle Gu,Au\right\rangle
	=
	0.
\end{equation}
$\left\langle Gu,Au\right\rangle=0$ by \eqref{Skiba-estimate}, \eqref{strongerthanskiba}, and \eqref{Calc3}.
Thus for any $t\geq \frac{1}{2}$ the perturbation satisfies
\begin{equation}
	\begin{split}
	& \dfrac{1}{2}\partial_{t}\norm{u}_{1}^{2}+\nu\norm{u}_{2}^{2}=0 \\
	\Rightarrow \
	& \norm{S_{t}u_{0}}_{H^{1}}^{2}\leq
	Ce^{-2\nu\lambda_{1}t}\norm{u(1/2)}_{H^{1}}^{2}.
	\end{split}
\end{equation}
Since $\norm{u'(1/2)}_{H^{1}}\leq \norm{u(1/2)}_{H^{1}}+\norm{\bar{u}(1/2)}_{H^{1}}$, \eqref{VtoHcontraction} gives
\begin{equation}
	\norm{S_{t}u_{0}}^{2}_{H} \leq \norm{S_{t}u_{0}}_{H^{1}}^{2}
	\leq
	C(\norm{u_{0}}_{H},\norm{\bar{u}_{0}}_{H})e^{-2\nu\lambda_{1}t}\norm{u_{0}}_{H}^{2}.
\end{equation}
Thus the solution is asymptotically attracting in $H$. \qed

\subsection{Proof of Lemma \ref{AlmostZonalProp}}

\noindent The proof uses a different approach than the analogous result in \cite{Varner}, Proposition 5.4.1. Instead we show that if the solution to the Navier-Stokes equations with a nonzonal force is ``close enough'' to the zonal solution, then it is globally exponentially stable. We then use standard estimates to the express the inequalities in terms of the distance from the force $f$. 

Let $u$ be the unique zonal solution for the Navier-Stokes equations with force $f$. Suppose $g$ is such that there exists $v=u+\bar{v}$ that solves
\begin{equation}
	\partial_{t}v+\nu Av +B(v,v) + C(v) = g.  \label{almostzonalnavierstokes}
\end{equation}
Let $\psi$ be another solution to \eqref{almostzonalnavierstokes} and consider $q= \psi - v$ which solves
\begin{equation}
	\partial_{t}q + \nu Aq + B(\psi,\psi) - B(v,v) + C(q) = 0.
\end{equation}
Rewriting the nonlinear terms gives
\begin{equation}
	\partial_{t}q + \nu Aq + B(q,q) + B(u,q)+ B(q,u)+B(\bar{v},q)+B(q,\bar{v}) +  C(q) = 0.
\end{equation}
Take the inner product with $Aq$. By \eqref{bproperty} $b(q,q,Aq) = 0$ and since $u$ is zonal $b(u,q,Aq) =0$ and $b(q,u,Aq)=0$ by \eqref{strongerthanskiba} and \eqref{Calc3}. By equations \eqref{bestimate3}, \eqref{bestimate4}, and the Cauchy inequality the trilinear terms $b(\bar{v},q,Aq)$ and $b(q,\bar{v},Aq)$ satisfy
\begin{align}
	b(\bar{v},q,Aq) + b(q,\bar{v},Aq)
	\leq & C\norm{\bar{v}}_{H^{1}}\norm{q}_{H^{1}}^{1/2}\norm{q}_{H^{2}}^{1/2}\norm{q}_{H^{2}} + C\norm{q}_{H^{1}}\norm{q}_{H^{2}}\norm{\bar{v}}_{H^{2}} \nonumber \\
	\leq & \frac{\nu}{4}\norm{q}_{H^{2}}^{2} + C\norm{\bar{v}}_{H^{1}}^{2}\norm{q}_{H^{1}}\norm{q}_{H^{2}} + C\norm{q}_{H^{1}}^{2}\norm{\bar{v}}_{H^{2}}^{2} \nonumber \\
	\leq & \frac{\nu}{2}\norm{q}_{H^{2}}^{2} + C\norm{q}_{H^{1}}^{2}\left(\norm{\bar{v}}_{H^{1}}^{4}+\norm{\bar{v}}_{H^{2}}^{2}\right). \nonumber
\end{align}
Thus 
\begin{equation}
			\partial_{t}\norm{q}_{H^{1}}^{2}
			\leq
			\norm{q}_{H^{1}}^{2}
			\left(-\lambda_{1}\nu 
			+C
			\left[\norm{\bar{v}}_{H^{1}}^{4}
			+\norm{\bar{v}}_{H^{2}}^{2}\right]\right).	\label{almostzonalrestriction}	
\end{equation}
For any $t>\frac{1}{2}$ integrating \eqref{almostzonalrestriction} on $\left[\frac{1}{2},t\right]$ gives 
\begin{equation}
			\norm{q(t)}_{H^{1}}^{2}
			\leq
			\norm{q\left(1/2\right)}_{H^{1}}^{2}
			exp\left(-\lambda_{1}\nu 
			+C
			\int_{1/2}^{t}\left[\norm{\bar{v}(\tau)}_{H^{1}}^{4}+ \norm{\bar{v}(\tau)}_{H^{2}}^{2}d\tau\right]\right).	
\end{equation}
Since $\norm{q\left(1/2\right)}_{H^{1}}^{2} \leq \norm{\psi\left(1/2\right)}_{H^{1}}^{2}+\norm{v\left(1/2\right)}_{H^{1}}^{2}$ by \eqref{VtoHcontraction}
\begin{equation} \label{almostzonalrestriction2}
	\begin{split}
		&	\norm{q(t)}_{H^{1}}^{2}
			\leq
			C(\norm{v_{0}}_{H},\norm{\psi_{0}}_{H}, \norm{f}_{L^{\infty}(0,T;H)}) \times \\
		&	exp\left(-\lambda_{1}\nu t
			+C
			\left[\int_{1/2}^{t}\norm{\bar{v}(\tau)}_{H^{1}}^{4}d\tau
			+\int_{1/2}^{t}\norm{\bar{v}(\tau)}_{H^{2}}^{2}d\tau\right]\right). \nonumber 
	\end{split}
\end{equation}
Thus if the norms of $\bar{v}$ are small enough then the unique solution $v$ is globally exponentially stable in $H^{1}$ (and thus in $H$).

It remains to express the norms of $\bar{v}$ in terms of the difference of forces. Since $\bar{v}=v-u$, consider the difference between the Navier-Stokes equations with force $f$ and zonal solution $u$ and equation \eqref{almostzonalnavierstokes} getting 
\begin{equation}
	\partial_{t}\bar{v} +\nu A\bar{v} - B(u,u)+B(v,v)+C(\bar{v})=f-g.
\end{equation}
Since 
$- B(u,u)+B(v,v)= B(u,\bar{v})+B(\bar{v},u)+B(\bar{v},\bar{v})$ and since $u$ is zonal, the inner product with $A\bar{v}$ and equations \eqref{bproperty}, \eqref{strongerthanskiba}, and \eqref{Calc3}  give
\begin{equation}
	\partial_{t}\norm{\bar{v}}_{H^{1}}^{2} +\nu \norm{\bar{v}}_{H^{2}}^{2}\leq C\norm{f-g}_{H}^{2}. \label{forwhatscoming}
\end{equation}
Integrating from $\left[\frac{1}{2},t\right]$, using $\norm{\bar{v}(1/2)}_{H^{1}}\leq \norm{u(1/2)}_{H^{1}}+\norm{v(1/2)}_{H^{1}}$, and \eqref{VtoHcontraction} gives 
\begin{equation}
	\int_{1/2}^{t}\norm{\bar{v}}_{H^{2}}^{2} \leq C\left(\norm{v_{0}}_{H},\norm{u_{0}}_{H}\right)+ C\norm{f-g}_{L^{\infty}(0,\infty;H)}^{2}t.
\end{equation}
Similarly using \eqref{Poincare} and integrating \eqref{forwhatscoming} from $\left[\frac{1}{2},t\right]$ yields
\begin{equation}
	\norm{\bar{v}(t)}_{H^{1}}^{2}\leq C(\norm{v_{0}}_{H},\norm{u_{0}}_{H})e^{-\lambda_{1}\nu t}+C\norm{f-g}_{L^{\infty}(0,\infty;H)}^{2}.
\end{equation}
Thus by Cauchy's inequality
\begin{equation}
	\norm{\bar{v}(t)}^{4}_{H^{1}}\leq Ce^{-2\lambda_{1}\nu t}+C\norm{f-g}_{L^{\infty}(0,\infty;H)}^{4}.
\end{equation}
Thus the exponential term in \eqref{almostzonalrestriction2} is bounded above by 
\begin{equation} \nonumber
			exp\left(C(\norm{v_{0}}_{H},\norm{u_{0}}_{H}) +Ct\left(-\lambda_{1}\nu+\norm{f-g}_{L^{\infty}(0,T;H)}^{4}+ \norm{f-g}_{L^{\infty}(0,\infty;H)}^{2}\right) \right). 
\end{equation}
Therefore there is $\delta>0$ such that if $\norm{f-g}_{L^{\infty}(0,\infty;H)}^{2}\leq \delta$ then the unique solution $v$ is globally exponentially stable in $H$. \qed

\bibliographystyle{amsplain}

\end{document}